\documentclass[11pt,a4paper]{article}

\usepackage{epsf,epsfig,amsfonts,amsgen,amsmath,amstext,amsbsy,amsopn,amsthm,cases,listings,color}
\usepackage{ebezier,eepic}
\usepackage{color}
\usepackage{multirow}
\usepackage{epstopdf}
\usepackage{graphicx}   
\usepackage{pgf,tikz}
\usepackage{mathrsfs}
\usepackage[marginal]{footmisc}
\usepackage{enumerate}
\usepackage{enumitem}
\usepackage[titletoc]{appendix}
\usepackage{booktabs}
\usepackage{url}
\usepackage{mathtools}
\usepackage{pdflscape} 
\usepackage{pgfplots}
\usepackage{authblk}
\usepackage{amssymb}
\usepackage{wasysym}
\usepackage{empheq}
\usepackage{dsfont}
\usepackage{tikz}
\usepackage{longtable}
\usepackage{float}
\usepackage{subfigure}
\pgfplotsset{compat=1.18}
\usepackage{mathrsfs}
\usepackage{wasysym} 
\usetikzlibrary{arrows}
\usepackage{aligned-overset}
\usepackage{bm}
\usepackage{bbm}
\usepackage[T1]{fontenc}
\usepackage{amsmath}
\usepackage{xcolor}
\usepackage[hidelinks]{hyperref}
\usepackage[nameinlink,noabbrev]{cleveref}

\allowdisplaybreaks[1]

\newtheorem{definition}{Definition} [section]
\newtheorem{theorem}[definition]{Theorem}

\newtheorem{proposition}[definition]{Proposition}

\newtheorem{claim}[definition]{Claim}
\newtheorem{problem}[definition]{Problem}

\newcommand{\one}{\mathbf 1}
\newcommand{\Frob}{\mathrm F}

\newcommand{\diag}{\operatorname{diag}}
\newcommand{\bip}{\mathrm{bip}}

\newcommand{\avd}{\overline d}
\newcommand{\navd}{\widetilde d}

\begin{document}
\title{\bf\Large Local versus global average degree in bipartite graphs\footnote{Research supported by National Key R\&D Program of China (Grant No. 2023YFA1010202). }}
\author{
Jianfeng Hou\thanks{
Email: jfhou@fzu.edu.cn},\quad \quad 
Hongbin Zhao\thanks{
Email: hbzhao2024@163.com }
\\
\small Center for Discrete Mathematics, Fuzhou University, Fujian, 350108, China
}
\date{}
\maketitle
\begin{abstract}
Let $F_{\mathrm{bip}}(n)$ denote the maximum, over all $n$-vertex bipartite graphs without isolated vertices, of the ratio of the minimum local average degree to the global average degree. We prove that $F_{\bip}(n)=\frac14\sqrt n+\frac38+o(1)$. 
This answers a problem posed by Tuza.

The lower bound is obtained from a new family of bipartite constructions, while the upper bound is proved by a new algebraic approach, different from the combinatorial methods used previously. The key ingredient is a spectral argument based on a normalized biadjacency matrix.

\medskip
\end{abstract}

\section{Introduction}

Throughout the paper, all graphs are assumed to be finite, simple, undirected, and without isolated vertices. Let $G$ be an $n$-vertex graph. For $v\in V(G)$, let $N_G(v)$ and
$d_G(v)$ denote the neighborhood and degree of $v$, respectively.
The \emph{global average degree} of $G$ and the \emph{local average
degree} of a vertex $v$ are defined by
\[
 \avd(G)=\frac{1}{n}\sum_{v\in V(G)}d_G(v)
       =\frac{2|E(G)|}{n},
 \qquad
 \navd_G(v)=\frac{1}{d_G(v)}
       \sum_{u\in N_G(v)}d_G(u).
\]
An elementary double-counting argument gives 
\begin{equation}\label{EQ:basic-global-local}
 \frac{1}{n}\sum_{v\in V(G)}\navd_G(v)-\avd(G)
 =\frac{1}{n}\sum_{uv\in E(G)}
   \frac{(d_G(u)-d_G(v))^2}{d_G(u)d_G(v)}\ge0.
\end{equation}
This implies that there exists $v\in V(G)$ such that
$\navd_G(v)\ge \avd(G)$. Such vertices, called \emph{groupies}, were introduced by Ajtai, Koml\'os, and Szemer\'edi~\cite{Ajtai1980}.
Bertram, Erd\H{o}s, Hor\'ak, \v{S}ir\'a\v{n}, and Tuza~\cite{BEHST}
initiated the study of groupies in the context of local and global average
degrees. Subsequent results investigated their abundance and structural
properties (see Mackey~\cite{Mackey1996}).

Another natural direction is to investigate the gap between local and global
average degrees. Define
\[
q(G)=\min_{v\in V(G)}\navd_G(v).
\]
Bertram, Erd\H{o}s, Hor\'ak, \v{S}ir\'a\v{n}, and Tuza~\cite{BEHST}
initiated the study of this extremal problem by proving that, among all
$n$-vertex graphs without isolated vertices,
\[
 \max_G
 \frac{q(G)}
      {\overline d(G)}
 =
 \frac14\sqrt{2n}+O(1).
\]
Poljak, Szab\'o, and Tuza~\cite{PST} subsequently obtained stronger bounds for this extremal ratio for bipartite graphs and trees.

In this paper, we focus on bipartite graphs. Define
\[
 F_{\bip}(n)=
 \max_{\substack{|V(G)|=n\\G\text{ bipartite, no isolates}}}
 \frac{q(G)}{\avd(G)}.
\]
As a natural candidate for the extremal construction, Tuza~\cite{TuzaProblems}
considered the following graph: take a complete bipartite graph $K_{p,p}$,
where $2p=\Theta(\sqrt n)$, as the core, and attach each of the remaining
vertices as a leaf to a core vertex, distributing these pendant vertices as
evenly as possible. This construction yields
\[
F_{\bip}(n)\ge \frac14\sqrt n-1.
\]
Tuza~\cite[Problem~17]{TuzaProblems} asked whether this construction is
asymptotically extremal.

\begin{problem}
Determine the asymptotic behavior of $F_{\bip}(n)$. In particular, is
\[
F_{\bip}(n)=\frac14\sqrt n+O(1)?
\]
\end{problem}

We resolve this problem by determining the sharp asymptotic expansion of
$F_{\bip}(n)$.

\begin{theorem}\label[theorem]{thm:main}
We have  $ F_{\bip}(n)=\frac14\sqrt n+\frac38+o(1)$.
\end{theorem}

Unlike previous combinatorial approaches based on degree counting and double
counting arguments, our proof of the upper bound introduces a new algebraic approach. The key idea is to associate the graph with a
normalized biadjacency matrix and use its singular values to translate
local average degree conditions into spectral constraints. Combined with Frobenius norm estimates, this converts the local average-degree condition into a global inequality involving the number of edges and the degree distribution. The proof of the lower bound is based on a refined extremal construction combined with a multi-phase arithmetic parameter scheme.

The paper is organized as follows. Section~\ref{sec:upper}  proves the upper bound via a singular-value argument. Section~\ref{sec:lower} gives the lower-bound construction.

\section{The upper bound via an algebraic approach}\label{sec:upper}

In this section, we prove the upper bound of $F_{\bip}(n)$ using an algebraic approach. For a real matrix $A=(a_{ij})$, write $\sigma_1(A)\ge\sigma_2(A)\ge\cdots\ge0$ for its singular values in non-increasing order, and $\sigma_{\min}(A)$ for its minimum singular value.
For a matrix with zero rows or zero columns, we adopt the convention $\sigma_1(A)=0$.
The \emph{Frobenius norm} of $A$ is defined by
\[
\|A\|_{\mathrm{F}}
=
\left(
\sum_{i,j}a_{ij}^{\,2}
\right)^{1/2}.
\]

\begin{theorem}\label{THM:upper-bound-main-result}
We have  
\[
\limsup\limits_{n\to\infty} \left( F_{\bip}(n)-\frac14\sqrt n \right) \le \frac38.
\]
\end{theorem}

\begin{proof}[Proof of Theorem~\ref{THM:upper-bound-main-result}]
Let $G=(X,Y;E)$ be a bipartite graph with no isolated vertices, and write
\[
n=|V(G)|,\qquad m=|E|,\qquad q=q(G).
\]
We may assume that $q\ge2$. Indeed, if $q<2$, then since
$\avd(G)\ge1$, we have $\frac{q}{\avd(G)}<2,$
which is already smaller than the desired upper bound for sufficiently large
$n$. The normalized biadjacency matrix of $G$ is the matrix
$\widehat C=(\widehat C_{xy})_{x\in X,y\in Y}$ defined by
\[
\widehat C_{xy}
=
\frac{\one_{xy\in E}}
{\sqrt{d(x)d(y)}}, 
\]
where 
\[
\one_{xy\in E}=
\begin{cases}
1, & \text{if }xy\in E(G),\\
0, & \text{if }xy\notin E(G).
\end{cases}
\]
is  the indicator function of the edge $xy$.

Let
\[
 X_2=\{x\in X:d(x)\ge2\},
 \qquad
 Y_2=\{y\in Y:d(y)\ge2\}.
\]
We refer to $G[X_2,Y_2]$ as the \emph{non-leaf core}.  We may assume that $X_2$ and $Y_2$ are both non-empty. 
Otherwise, say $X_2=\emptyset$, this means  every vertex in $X$ has degree $1$. Thus,  $q(G)=1$ and  the case is trivial.

Define matrices $C=(C_{xy})_{x\in X_2,y\in Y_2}$ and $B=(B_{xy})_{x\in X_2,y\in Y_2}$ by
\[
 C_{xy}=\frac{\one_{xy\in E}}{\sqrt{d(x)d(y)}},
 \qquad
 B_{xy}=\one_{xy\in E}
 \sqrt{\frac{(d(x)-1)(d(y)-1)}{d(x)d(y)}}. 
\]
Clearly, $C$ is a submatrix of $\widehat C$. Let 
\[
P_X=\diag\bigl(\sqrt{d(x)-1}:x\in X_2\bigr) \quad \text{ and }  \quad P_Y=\diag\bigl(\sqrt{d(y)-1}:y\in Y_2\bigr).
\]
Clearly, 
\begin{equation}\label{EQ:relation-B-C}
B=P_XCP_Y
\end{equation}
The role of \(B\) is to encode the excess degrees \(d(v)-1\), so that the condition \(q(G)\ge q\) yields a lower bound on \(\sigma_1(B)\). In contrast, \(C\) is a submatrix of the normalized biadjacency matrix and hence satisfies \(\sigma_1(C)\le1\). We first compare the largest singular values of $B$ and $C$.

\begin{proposition}\label{prop:basic-spectral}
We have  $\sigma_1(C)\le1$ and $\sigma_1(B)\ge q-1$.
\end{proposition}

\begin{proof}
Recall that  $C$ is a submatrix of $\widehat C$. To prove $\sigma_1(C)\le1$, it suffices to show $\sigma_1(\widehat C)\le 1$. 
For any real vectors $r=(r_x)_{x\in X}$ and $s=(s_y)_{y\in Y}$, we have
\[
2r^{\mathsf T}\widehat C s
=
2\sum_{xy\in E}
\frac{r_xs_y}{\sqrt{d(x)d(y)}}.
\]
Applying the inequality $2ab\le a^2+b^2$ to each edge term gives
\[
2r^{\mathsf T}\widehat C s
\le
\sum_{xy\in E}
\left(
\frac{r_x^2}{d(x)}
+
\frac{s_y^2}{d(y)}
\right)
=\sum_{x\in X}r_x^2+\sum_{y\in Y}s_y^2=\|r\|_2^2+\|s\|_2^2.
\]
Applying the same inequality with $-r$ in place of $r$, we obtain
\[
2\lvert r^{\mathsf T}\widehat C s\rvert
\le
\|r\|_2^2+\|s\|_2^2.
\]
Recall that
\[
\sigma_1(\widehat C)
=
\max_{\|r\|_2=\|s\|_2=1}
|r^{\mathsf T}\widehat C s|.
\]
Thus, $\sigma_1(\widehat C)\le 1$.

Now we prove $\sigma_1(B)\ge q-1$. 
Define
$u\in\mathbb R^{X_2}$ and $v\in\mathbb R^{Y_2}$ by 
\[
u_x=\sqrt{d(x)(d(x)-1)} \quad \text{ and } \quad v_y=\sqrt{d(y)(d(y)-1)}.
\]
By the definition of $B$, for $x\in X_2$, we have 
\begin{align*}
    (Bv)_x
    &=
\sum_{\substack{y\in Y_2\\xy\in E}}
\sqrt{
\frac{(d(x)-1)(d(y)-1)}
{d(x)d(y)}
}
\sqrt{d(y)(d(y)-1)}\\
 &=\sqrt{\frac{d(x)-1}{d(x)}}
   \sum_{\substack{y\in Y_2\\xy\in E}}(d(y)-1)\\
   & =\sqrt{\frac{d(x)-1}{d(x)}}
   \sum_{y\in N(x)}(d(y)-1)\\
 &\ge(q-1)\sqrt{d(x)(d(x)-1)}
 =(q-1)u_x.
\end{align*}
This implies that 
\begin{equation}\label{EQ:Bv-(q-i)u}
Bv\ge (q-1)u. 
\end{equation}
Similarly, 
\begin{equation}\label{EQ:BTu-(q-i)v}
B^{\mathsf T}u\ge(q-1)v. 
\end{equation}

Set
\[
\mathcal{A}= \begin{pmatrix}0&B\\B^{\mathsf T}&0\end{pmatrix} \quad \text{ and } \quad w = \binom{u}{v}. 
\]
By \eqref{EQ:Bv-(q-i)u} and \eqref{EQ:BTu-(q-i)v}, we have  $\mathcal{A}w\ge (q-1)w$ entrywise. Since $w$ is non-zero and non-negative, the Rayleigh quotient gives 
\[
\lambda_{max}(\mathcal{A})\ge \frac{w^T\mathcal{A}w}{w^Tw}\ge q-1. 
\]
Note that the eigenvalues of $\mathcal{A}$ are $\pm\sigma_i(B)$, together with possible zero eigenvalues. Hence $\sigma_1(B)\ge q-1$.
\end{proof}

Let 
\[
Z_\ell = \sum_{\substack{xy\in E\\\min\{d(x),d(y)\}=1}} \frac{1}{d(x)d(y)}
\] 
be the reciprocal mass of the leaf edges. Note that for  every leaf edge $xy$, say $d(x)=1$, we have $\navd_G(x)=d(y)\ge q,$
which implies that 
\[
\frac{1}{d(x)d(y)}=\frac1{d(y)}\le\frac1q.
\]
Therefore,
\begin{equation}\label{EQ:upper-bound-Z-l}
Z_\ell\le\frac{n}{q}.
\end{equation}
In the following, we give a lower bound on $Z_\ell$.

\begin{proposition}\label{prop:spectral-leaf}
We have $m-n+Z_\ell\ge q(q-2)$.
\end{proposition}

\begin{proof}

By the definition of $P_X$, every diagonal entry of $P_X$ is of the form
$\sqrt{d(x)-1}$ with $x\in X_2$. Since $d(x)\ge2$ for every $x\in X_2$, we have
\[
\sigma_{\min}(P_X)=\min_{x\in X_2}\sqrt{d(x)-1}\ge1.
\]
Similarly, $\sigma_{\min}(P_Y)\ge1$. Thus, by \eqref{EQ:relation-B-C}, we have
\begin{equation}\label{EQ:bound-sigmaB-sigmaC}
\sigma_i(B)
\ge
\sigma_{\min}(P_X)\sigma_i(C)\sigma_{\min}(P_Y)
\ge
\sigma_i(C)
\qquad\text{for }1\le i \le \min\{|X_2|,|Y_2|\}.
\end{equation}

On the other hand, we have 
\begin{align}\label{EQ:F-norm-B}
\|B\|_{\mathrm{F}}^2
&=\sum_{xy\in E}
\frac{(d(x)-1)(d(y)-1)}{d(x)d(y)}\notag\\
&=\sum_{xy\in E}
\left(
1-\frac{1}{d(x)}-\frac{1}{d(y)}
+\frac{1}{d(x)d(y)}
\right)\notag\\
&=m-|X|-|Y|
+\sum_{xy\in E}\frac{1}{d(x)d(y)}\notag\\
&=m-n+\sum_{xy\in E}\frac{1}{d(x)d(y)}.
\end{align}
and 
\begin{align}\label{EQ:F-norm-C}
\|C\|_{\Frob}^2
=\sum_{x\in X_2}\sum_{y\in Y_2}
\frac{\mathbf 1_{xy\in E}}
{d(x)d(y)}=\sum_{\substack{xy\in E\\ d(x),d(y)\ge2}}
\frac1{d(x)d(y)} .
\end{align}
Combining \eqref{EQ:bound-sigmaB-sigmaC}, \eqref{EQ:F-norm-B}, \eqref{EQ:F-norm-C}, and Proposition \ref{prop:basic-spectral}, we conclude that 
\[
m-n+Z_\ell
=\|B\|_{\Frob}^2-\|C\|_{\Frob}^2
=\sum_i\bigl(\sigma_i(B)^2-\sigma_i(C)^2\bigr)
\ge \sigma_1(B)^2-\sigma_1(C)^2\ge q(q-2). 
\]
This completes the proof. 
\end{proof}

Combining \eqref{EQ:upper-bound-Z-l} and Proposition~\ref{prop:spectral-leaf}, we have $m\ge n+q^2-2q-\frac nq$. 
Consequently,
\begin{equation}\label{EQ:aim-inequality}
 \frac{q}{\avd(G)}=\frac{nq}{2m}
 \le
 \frac{nq}{2\left(n+q^2-2q-n/q\right)}.
\end{equation}

Put $N=\sqrt n$ and define
\[
f_N(q):=
\frac{N^2q}
{2\left(N^2+q^2-2q-N^2/q\right)},
\qquad
h_N(q):=q^3-N^2q+2N^2.
\]
To complete the proof, by \eqref{EQ:aim-inequality}, it suffices to show 
\[
\max_{q\ge2} f_N(q)-\frac{N}{4}\longrightarrow \frac{3}{8}, 
\]
as $N\longrightarrow \infty$. Note that a direct differentiation gives
\[
f_N'(q)
=
-\frac{N^2h_N(q)}
{2q\left(N^2+q^2-2q-N^2/q\right)^2},
\]
and hence
\begin{equation}\label{EQ:sgn-f'N-hN}
\operatorname{sgn} f_N'(q)
=
-\operatorname{sgn} h_N(q).
\end{equation}

On the other hand, for all sufficiently large $N$,
\[
h_N(-N-1)<0<h_N(-N),\qquad
h_N(2)>0>h_N(3),\qquad
h_N(N-2)<0<h_N(N).
\]
Thus $h_N$ has three real roots, lying respectively in
\[
(-N-1,-N),\qquad (2,3),\qquad (N-2,N).
\]
Let $q_N\in (N-2,N)$ and  $r_N\in(2,3)$ denote the  roots of $h_N$. By \eqref{EQ:sgn-f'N-hN}, we know that 
$f_N$ is decreasing on $[2,r_N]$, increasing on $[r_N,q_N]$, and decreasing
on $[q_N,\infty)$.
Consequently,
\[
\max_{q\ge2} f_N(q)
=
\max\{f_N(2),f_N(q_N)\}.
\]
Recall that $q_N/N\longrightarrow1.$ Consequently,
\[
f_N(q_N)\sim\frac N4,
\qquad
f_N(2)=2.
\]
This means that $\max_{q\ge2}f_N(q)=f_N(q_N)$ as $N\longrightarrow \infty$.

Write $q_N=N+s_N.$
It follows from $h_N(q_N)=0$ that 
\[
2(s_N+1)+\frac{3s_N^2}{N}+\frac{s_N^3}{N^2}=0.
\]
It follows from  $-2<s_N<0$ that  $s_N\longrightarrow-1.$
A direct calculation gives
\[
\begin{aligned}
f_N(q_N)-\frac N4
=
\frac{
2q_N/N-s_N^2/N+N/q_N
}{
4\left(
1+q_N^2/N^2-2q_N/N^2-1/q_N
\right)
}
&\longrightarrow
\frac{2+1}{4(1+1)}
=\frac38.
\end{aligned}
\]
This completes the proof. 
\end{proof}



%
\section{The lower-bound construction}\label{sec:lower}

In this section, we prove the lower bound using a family of complete bipartite cores with leaves attached to both sides. 
In the unrestricted setting, Bertram, Erd\H{o}s, Hor\'ak, \v{S}ir\'a\v{n}, and Tuza obtained the optimal order of magnitude by considering a dense clique together with a large collection of pendant vertices. 
Our construction replaces the clique core by a carefully balanced complete bipartite core with appropriately distributed pendant vertices. 
However, a direct bipartite analogue is not sufficient to obtain the sharp constant term, since the parameters of the bipartite core must simultaneously satisfy the bipartition constraint and the prescribed number of vertices. 
To overcome this difficulty, we introduce a refined arithmetic construction with several parameter phases, which guarantees coverage of all sufficiently large orders while preserving the required lower bound on the minimum local average degree.

\begin{theorem}\label{THM:lower-bound-main}
For every $\varepsilon>0$, there exists an integer $N_0$ such that, for
every $n\ge N_0$, there is an $n$-vertex bipartite graph $G$ satisfying
\[
\frac{q(G)}{\avd(G)}
\ge
\frac14\sqrt n+\frac38-\varepsilon.
\]
\end{theorem}

Clearly, Theorem \ref{thm:main} follows from Theorem \ref{THM:upper-bound-main-result} and Theorem \ref{THM:lower-bound-main}.

\begin{proof}[Proof of Theorem~\ref{THM:lower-bound-main}]

Choose an integer $M\ge2$ such that $\frac1{8(2M-1)}<\frac\varepsilon2$, and put 
\[
P=2M-1,~ L=MP(3M-1).
\]
We use $2MP$ phases.  For a phase $j\in\{0,1,\ldots,2MP-1\}$ and a sufficiently large integer $k$, set
\[
 s=\frac{j}{2MP},
 \qquad
 Q=Lk+s.
\]
We remark that  $k$ is the main asymptotic parameter; the other parameters serve to
control the target local average degree $Q$ and handle the
arithmetic and covering issues needed to realize every sufficiently large
integer $n$.

We take a complete bipartite core $K_{a,b}$ and attach $h$ private leaves to every vertex on the $a$-side and $g$ private leaves to every vertex on the $b$-side, where
\[
\begin{aligned}
 a&=P^2(3M-1)k+\xi_a,
 & b&=M^2(3M-1)k+\xi_b,\\
 h&=M^2(3M-2)k+\xi_h,
 & g&=P^2k+\xi_g.
\end{aligned}
\]
The integer offsets $\xi_a,\xi_b,\xi_h,\xi_g$, to be specified below, depend on the phase $j$ and on an auxiliary integer $t$, introduced below, but not on $k$.

Let $G_0$ denote the resulting graph, and call the vertices of $K_{a,b}$ the core vertices. Then
\[
 |V(G_0)|=a+b+ah+bg=\Theta_M(Q^2),
 \qquad
 |E(G_0)|=ab+ah+bg.
\]
Clearly, 
\begin{equation}\label{EQ:E-V=ab-a-b}
  |E(G_0)|-|V(G_0)|=ab-a-b.   
\end{equation}

We now compute the local average degrees of the four types of vertices in $G_0$. Define
\[
 S_X:=ah-(Q-1)g+a(b-Q),
 \qquad
 S_Y:=bg-(Q-1)h+b(a-Q).
\]
Let $x$ be a core vertex on the $a$-side. Its $b$ core neighbors each have degree $a+g$, while its $h$ pendant neighbors have degree one.  A direct calculation yields that 
\begin{equation}\label{EQ:local-ave-G0-a-side}
 \navd_{G_0}(x)
 =
 \frac{b(a+g)+h}{b+h}
 =
 Q+\frac{S_Y}{b+h}.
\end{equation}
Similarly, if $y$ is a core vertex on the $b$-side, then
\begin{equation}\label{EQ:local-ave-G0-b-side}
 \navd_{G_0}(y)
 =
 \frac{a(b+h)+g}{a+g}
 =
 Q+\frac{S_X}{a+g}.
\end{equation}

Now let $u$ be a pendant vertex adjacent to an $a$-side core vertex. Since its unique neighbor has degree $b+h$,
\begin{equation}\label{EQ:local-ave-G0-a-pendant}
 \navd_{G_0}(u)
 =
 b+h
 =
 Q+MPk+O_{M,t}(1).
\end{equation}
Likewise, if $v$ is a pendant vertex adjacent to a $b$-side core vertex, then
\begin{equation}\label{EQ:local-ave-G0-b-pendant}
 \navd_{G_0}(v)
 =
 a+g
 =
 Q+MP(3M-2)k+O_{M,t}(1).
\end{equation}

By \eqref{EQ:local-ave-G0-a-side}, \eqref{EQ:local-ave-G0-b-side}, \eqref{EQ:local-ave-G0-a-pendant} and \eqref{EQ:local-ave-G0-b-pendant}, $S_X$ and $S_Y$ measure the excess of the local average degrees of the core vertices above  $Q$, while the pendant vertices have local average degree exceeding $Q$ by an amount of order $k$.

Define 
\[
r=M\xi_a+P(3M-1)(\xi_b+\xi_h)-M(3M-1)\xi_g-3MPs+P. 
\]
The following arithmetic claim provides the required choice of offsets.

\begin{claim}\label{clm:offsets}
For each phase $j$ and each integer $t$, there exist integers
$\xi_a,\xi_b,\xi_h,\xi_g$ such that
\begin{equation}\label{EQ:choice-A-B-H-G}
   M^2\xi_a+P^2\xi_b=j-M^2+1,
 \qquad
 r\in \{1,3/2\},
 \qquad
 \xi_a+\xi_b+\xi_g+\xi_h=t-2.  
\end{equation}
\end{claim}

\begin{proof}
First choose any integers $\xi_a,\xi_b$ satisfying
\begin{equation}\label{eq:linear-condition}
 M^2\xi_a+P^2\xi_b=j-M^2+1.
\end{equation}
This is possible since $\gcd(M^2,P^2)=1$.

We now show that the remaining freedom in the offsets allows us to arrange  $r\in \{1,3/2\}$.  On the one hand, 
replacing $(\xi_a,\xi_b)$ by $(\xi_a+P^2,\xi_b-M^2)$ preserves \eqref{eq:linear-condition} and changes $r$ by $-MP(3M^2-3M+1)$. On the other hand, the dependence of $r$ on $\xi_h$ and $\xi_g$ is through
$(3M-1)(P\xi_h-M\xi_g)$.  Since $\gcd(P,M)=1$, we have that 
$P\xi_h-M\xi_g$ ranges over all integers as $\xi_h,\xi_g\in\mathbb Z$.  Thus varying
$\xi_h$ and $\xi_g$ changes $r$ by an arbitrary multiple of $3M-1$.

We claim $\gcd\bigl(MP(3M^2-3M+1),3M-1\bigr)=1$. Indeed, $\gcd(M,3M-1)=1$, and $\gcd(P,3M-1)=1$. Moreover, $3M^2-3M+1+P=M(3M-1)$, and hence $3M^2-3M+1\equiv -P \pmod{3M-1}$. 
Thus $\gcd(3M^2-3M+1,3M-1)=1$ and we are done.

If $r_0$ is one attainable value, then  the full set of attainable values is 
\begin{align*}
    r_0+MP(3M^2-3M+1)\mathbb Z+(3M-1)\mathbb Z = r_0+\mathbb Z.
\end{align*}
Thus one may choose the offsets so that $1\le r<2$.
Moreover, all terms in the definition of $r$ except $-3MPs=-3j/2$ are integers. Thus this choice necessarily gives $r=1$ when $j$ is even and $r=3/2$ when $j$ is odd.

It remains to adjust $\xi_a+\xi_b+\xi_g+\xi_h$ without changing either \eqref{eq:linear-condition} or $r$.  The move
\[
(\xi_a,\xi_b,\xi_h,\xi_g)\mapsto (\xi_a,\xi_b,\xi_h+M,\xi_g+P)
\]
preserves both quantities and changes $\xi_a+\xi_b+\xi_g+\xi_h$ by $3M-1$.  Also the move
\[
(\xi_a,\xi_b,\xi_h,\xi_g)\mapsto
\bigl(\xi_a+P^2(3M-1),\,\xi_b-M^2(3M-1),\,\xi_h,\,
\xi_g-P(3M^2-3M+1)\bigr)
\]
preserves both quantities and changes $\xi_a+\xi_b+\xi_g+\xi_h$ by $3M^3-6M^2+2M$.
Therefore, after \eqref{eq:linear-condition} and $r$ have been fixed, the sum $\xi_a+\xi_b+\xi_g+\xi_h$ can still be changed by any integer combination of $3M-1$ and $3M^3-6M^2+2M$.
Since $9(3M^3-6M^2+2M)\equiv 1 \pmod {3M-1}$, we have $\gcd(3M-1,~3M^3-6M^2+2M)=1$.  Hence these two moves allow us to change \(\xi_a+\xi_b+\xi_g+\xi_h\) by an arbitrary integer, and we may therefore arrange that \(\xi_a+\xi_b+\xi_g+\xi_h=t-2\).
\end{proof}

For each pair $(j,t)$, fix a choice of $\xi_a,\xi_b,\xi_h,\xi_g$  given by Claim \ref{clm:offsets}.  We remark that the three conditions in \eqref{EQ:choice-A-B-H-G} play different roles in the construction. The first condition is used to control the linear term in the expansion of $ab-a-b$. Indeed, expanding in powers of $k$ and using the first equation in \eqref{EQ:choice-A-B-H-G}, we obtain
\begin{align}\label{EQ:ab-a-b}
    ab-a-b
    &=L^2k^2+(3M-1)(M^2\xi_a+P^2\xi_b-P^2-M^2)k+O_{M,t}(1)\notag\\
    &=Q^2-\left(3-\frac{1}{P}\right)Q+O_{M,t}(1).
\end{align}

Next, collecting the coefficients of $k$ in $S_X$ and $S_Y$ and again using the first equation in \eqref{EQ:choice-A-B-H-G}, we obtain
\begin{equation}\label{eq:capacities}
 S_X=Prk+O_{M,t}(1),
 \qquad
 S_Y=M(2-r)k+O_{M,t}(1).
\end{equation}
The second condition, $r\in\{1,3/2\}$, therefore ensures that both $S_X$ and $S_Y$ are positive for all sufficiently large $k$, providing enough room to delete pendant vertices from either side while maintaining the required lower bound on the local average degrees.

Finally, the last condition in \eqref{EQ:choice-A-B-H-G} provides the additional freedom needed to shift the interval of realizable graph orders. This will be discussed in detail later.

We next record the deletion rule, now in the form in which it will be applied to the above core.  

\begin{claim}\label{clm:deletion-rule}
Assume $S_X,S_Y\ge0$.  If an integer $D$ satisfies
\[
 \lfloor S_X\rfloor\le aD,
 \qquad
 \lfloor S_Y\rfloor\le bD,
 \qquad
 D\le \min\{h,g,b+h-Q,a+g-Q\},
\]
then, for every integer $N\in\bigl[a+b+ah+bg-\lfloor S_X\rfloor-\lfloor S_Y\rfloor,\,a+b+ah+bg\bigr]$,  there exists a graph $G$, obtained from $G_0$ by deleting at most $D$ leaves at each core vertex, such that $|V(G)|=N$ and $q(G)\ge Q$.
\end{claim}

\begin{proof}
Put $d=a+b+ah+bg-N$.  Then
$0\le d\le \lfloor S_X\rfloor+\lfloor S_Y\rfloor$.  Choose $d_X,d_Y$ such that
\[
 d=d_X+d_Y,
 \qquad
 0\le d_X\le \lfloor S_X\rfloor,
 \qquad
 0\le d_Y\le \lfloor S_Y\rfloor.
\]
Delete $d_X$ leaves adjacent to the $a$-side core vertices and $d_Y$ leaves adjacent to the $b$-side core vertices.  Since $\lfloor S_X\rfloor\le aD$, $\lfloor S_Y\rfloor\le bD$, and $D\le\min\{h,g\}$, these deletions can be distributed so that at most $D$ leaves are deleted at each core vertex.  The resulting graph has exactly $N$ vertices.

We check the local averages.  Let $x$ be an $a$-side core vertex, and let $\lambda_x$ be the number of leaves still attached to $x$.  The total number of remaining leaves attached to the $b$-side core vertices is $bg-d_Y$, so
\[
 \navd_G(x)=\frac{ab+bg-d_Y+\lambda_x}{b+\lambda_x}.
\]
Using \(d_Y\le S_Y\), the choice of the parameters, and
\(\lambda_x\le h\), we obtain
\[
\begin{aligned}
bg-d_Y+b(a-Q)
\ge bg-S_Y+b(a-Q)
\ge (Q-1)h
\ge (Q-1)\lambda_x.
\end{aligned}
\]
This implies that  $\navd_G(x)\ge Q$.
Similarly, if $y$ is a $b$-side core vertex and $\lambda_y$ leaves remain attached to it, then
\[
 \navd_G(y)=\frac{ab+ah-d_X+\lambda_y}{a+\lambda_y}\ge Q.
\]
Finally, every remaining leaf is adjacent to a core vertex of degree at least $b+h-D$ or $a+g-D$, both of which are at least $Q$ by \eqref{EQ:local-ave-G0-a-pendant} and \eqref{EQ:local-ave-G0-b-pendant}.  Thus $q(G)\ge Q$.
\end{proof}

Let
\[
 \alpha_M=P(3M^2-4M+2)-(3M-2).
\]
Collecting the quadratic and linear terms in $k$ and using \eqref{EQ:choice-A-B-H-G}, we obtain
\begin{equation}\label{EQ:N-max-Q2}
   N_{\max}:=a+b+ah+bg
 =Q^2+\bigl(\alpha_M+(M-1)r+Lt-3Ls\bigr)k+O_{M,t}(1).  
\end{equation}
Using \eqref{eq:capacities} and the fact $Pr+M(2-r)=(M-1)r+2M$, we have 
\begin{align}\label{EQ:N-min-Q2}
 N_{\min}:=N_{\max}-\lfloor S_X\rfloor-\lfloor S_Y\rfloor =Q^2+\bigl(\alpha_M+Lt-3Ls-2M\bigr)k+O_{M,t}(1). 
\end{align}
For fixed $M$ and $t$, it follows from  \(Q=Lk+O_{M,t}(1)\) that \(Q\asymp_M k\), and then 
\[
\frac{\beta k+O_{M,t}(1)}{Q^2}=O_{M,t}(k^{-1}). 
\]
By the expansion \(\sqrt{1+x}=1+x/2+O(x^2)\), we have 
\[
\begin{aligned}
\sqrt{Q^2+\beta k+O_{M,t}(1)}
&=
Q\sqrt{1+\frac{\beta k+O_{M,t}(1)}{Q^2}}
&=
Q+\frac{\beta}{2L}+O_{M,t}(k^{-1}). 
\end{aligned}
\]
Together with \eqref{EQ:N-max-Q2} and \eqref{EQ:N-min-Q2}, this shows that if \(N\in[N_{\min},N_{\max}]\), then \(\sqrt N\) lies in the following interval, with each endpoint shifted inward by \(O_{M,t}(k^{-1})\): 
\begin{equation}\label{eq:root-interval}
 Q+
 \left[
 \frac{\alpha_M+Lt-3Ls-2M}{2L},
 \frac{\alpha_M+(M-1)r+Lt-3Ls}{2L}
 \right].
\end{equation}
Note that increasing $t$ by one translates this interval by $1/2$.

\begin{claim}\label{clm:circle-covering}
As the phase $j$ ranges over $0,1,\ldots,2MP-1$, the intervals in \eqref{eq:root-interval} cover the circle
\(\mathbb R/(\frac12\mathbb Z)\) with uniformly positive overlap depending only on \(M\).
\end{claim}

\begin{proof}
    Fix $t$. For the interval in \eqref{eq:root-interval} corresponding to phase $j$, let 
   \[
   R_j
=
Q+\frac{\alpha_M+(M-1)r+Lt-3Ls}{2L}.
   \]  
   be its right endpoint. We shall consider these $R_j$ modulo \(1/2\). Since \(Q=Lk+s\), \(s=j/(2MP)\), and \(Lt/(2L)=t/2\), we have
    $$R_j\equiv \frac{\alpha_M}{2L}-\frac{j}{4MP}+\frac{(M-1)r}{2L}\pmod{1/2}.$$
    Thus, modulo \(1/2\), \(R_j\) is obtained from
\[
R_j^0=\frac{\alpha_M}{2L}-\frac{j}{4MP},
\]
where the points \(R_j^0\) are equally spaced with spacing \(1/(4MP)\), by adding the perturbation
\((M-1)r/(2L)\). Since \(1\le r<2\), the perturbation lies in \([ (M-1)/(2L), (M-1)/L )\), whose width is \((M-1)/(2L)\). This width is smaller than the spacing, because
    \[
        \frac{(M-1)/(2L)}{1/(4MP)} = \frac{2(M-1)}{3M-1} <1.
    \]
This means that the perturbations do not change the cyclic order of the right endpoints. It follows that every gap between two consecutive right endpoints
is at most \(1/(4MP)+(M-1)/(2L)\).

On the other hand, the length of  the interval in \eqref{eq:root-interval} is 
\[
\frac{2M+(M-1)r}{2L}\ge \frac{3M-1}{2L} =\frac{1}{4MP}+ \frac{M-1}{2L}+ \frac{M+1}{4MP(3M-1)}. 
\]
Therefore, in the cyclic ordering of the right endpoints, each interval extends to the left beyond the previous right endpoint by at least $\eta_M:=\frac{M+1}{4MP(3M-1)}>0$.  Thus their interiors form an open cover of the circle $\mathbb{R}/\left(\frac12\mathbb{Z}\right)$.
By the Lebesgue number lemma, there exists \(\delta_M>0\) such that every
point of the circle has a neighborhood of radius \(\delta_M\) contained in
the interior of one of the intervals. Consequently, every point lies in one of the intervals at distance at least $\delta_M$ from both endpoints.

\end{proof}

\begin{claim}\label{clm:exist-j-t-for-n}
For every sufficiently large integer \(n\), there exist a phase \(j\) and an integer \(t\) so that  \( n\in [N_{\min}, N_{\max}]\). 
\end{claim}
\begin{proof}
For a sufficiently large integer \(n\), let $x=\sqrt n$, set $k=\lfloor x/L\rfloor$, and put $y=x-Lk\in[0,L)$. For each phase $j$, let $I_j$ be the interval obtained from \eqref{eq:root-interval} by taking $t=0$ and subtracting $Lk$ from both endpoints. Replacing $t$ by $t+1$ translates $I_j$ by exactly $1/2$. By Claim~\ref{clm:circle-covering}, the images of the intervals $I_j$ cover $\mathbb R/(\frac12\mathbb Z)$ with a uniform positive overlap. Hence there exist a phase $j$ and an integer $t$ such that $y\in I_j+t/2$, with $y$ at distance at least $\delta_M$ from both endpoints. 

Since $y\in[0,L)$ and the finitely many intervals $I_j$ lie in a fixed bounded set depending only on $M$, the integer $t$ may be chosen with $|t|\le T_M$ for some constant $T_M$ depending only on $M$.

We then verify that the accumulated approximation and rounding errors are smaller than the available margin, and hence the inclusion persists for the exact deletion interval. Since \(j\le 2MP-1\)  and \(t\) is uniformly bounded by \(T_M\), only finitely many pairs \((j,t)\) can occur. Hence the corresponding offsets \(\xi_a,\xi_b,\xi_h,\xi_g\) are uniformly \(O_M(1)\), and all error terms \(O_{M,t}(1)\) in  \eqref{eq:capacities} and \(O_{M,t}(k^{-1})\) in \eqref{eq:root-interval} may be replaced by \(O_M(1)\) and \(O_M(k^{-1})\). By the preceding covering argument, the chosen point lies at distance at least \(\delta_M\) from both endpoints of the limiting interval. The exact endpoints differ from these limiting endpoints by \(O_M(k^{-1})\), including the errors arising from the floor functions. Thus, for sufficiently large \(k\), these perturbations are smaller than \(\delta_M\), so the point remains inside the exact interval. Therefore \(n\) lies in the exact deletion interval of Claim~\ref{clm:deletion-rule}.
\end{proof}

It remains only to check the bounded-deletion condition in Claim~\ref{clm:deletion-rule}. Recall that there are finitely many pairs  \((j,t)\). After fixing one choice
of offsets from Claim~\ref{clm:offsets} for each such pair \((j,t)\),
we have $ \xi_a,\xi_b,\xi_h,\xi_g=O_M(1)$ uniformly.  In particular, by \eqref{eq:capacities},
\[
 \frac{S_X}{a}=O_M(1),
 \qquad
 \frac{S_Y}{b}=O_M(1)
\]
uniformly over all  \((j,t)\).  We may therefore choose
an integer \(D=D(M)\) such that
\[
 \lfloor S_X\rfloor\le aD,
 \qquad
 \lfloor S_Y\rfloor\le bD
\]
for every such pair. Moreover,
\[
 h,g=\Omega_M(k),  \qquad b+h-Q=MPk+O_M(1),
 \qquad
 a+g-Q=MP(3M-2)k+O_M(1).
\]
Hence, for all sufficiently large \(k\),
\[
 D\le \min\{h,g,b+h-Q,a+g-Q\}.
\]
Thus,  Claim~\ref{clm:deletion-rule} can be applied uniformly to every pair \((j,t)\).

For every sufficiently large $n$, using Claims \ref{clm:deletion-rule}, \ref{clm:exist-j-t-for-n} and the above arguments yields that there exists an $n$-vertex graph $G$ with $q(G)\ge Q$. To end the proof, we bound $q(G)/\avd(G)$. By \eqref{eq:root-interval},  $Q-\sqrt n=O_M(1)$, which implies that 
\begin{equation}\label{EQ:Q2-n}
  Q^2=n+2(Q-\sqrt n)\sqrt n+O_M(1).  
\end{equation}
Assume that $G$ has $m$ edges.  Recall that  $G$ is obtained from $G_0$ by deleting leaves adjacent to core vertices by Claim \ref{clm:deletion-rule}. Thus 
\[
m-n=|E(G_0)|-|V(G_0)|=ab-a-b, 
\]
where the last equation follows from \eqref{EQ:E-V=ab-a-b}. This together with \eqref{EQ:ab-a-b} gives  that 
\begin{equation}\label{EQ:m=n+ab-a-b}
m=n+Q^2-\left(3-\frac{1}{P}\right)Q+O_{M,t}(1). 
\end{equation}
Combining \eqref{EQ:Q2-n} and \eqref{EQ:m=n+ab-a-b}, we have 
\begin{align*}
m=2n\left(
1+\frac{Q-\sqrt n-(3-1/P)/2}{\sqrt n}
+O_M(n^{-1})
\right).
\end{align*}
Consequently, 
\begin{align*}
 \frac{q(G)}{\avd(G)}
 &\ge \frac{nQ}{2m} = \frac{\sqrt n}{4}\times 
 \frac{1+(Q-\sqrt n)/\sqrt n}
 {1+\bigl(Q-\sqrt n-(3-1/P)/2\bigr)/\sqrt n+O_M(n^{-1})}\\
 &=\frac{\sqrt n}{4}
 \left(1+\frac{Q-\sqrt n}{\sqrt n}\right)
 \left(1-\frac{Q-\sqrt n-(3-1/P)/2}{\sqrt n}
 +O_M(n^{-1})\right)\\
 &=\frac{\sqrt n}{4}
 \left(1+\frac{3-1/P}{2\sqrt n}
 +O_M(n^{-1})\right)\\
 &=\frac14\sqrt n+\frac38-\frac1{8(2M-1)}
 +O_M(n^{-1/2})\\
 &>\frac14\sqrt n+\frac38-\varepsilon.
\end{align*}
Note that the last inequality follows from the fact that $\frac1{8(2M-1)}<\frac\varepsilon2$ and $n$ is large. 
This completes the proof. 
\end{proof}

\section{Concluding remarks}\label{sec:concluding}

Theorem~\ref{thm:main} settles the asymptotic problem posed by Tuza, including the additive constant. 
A main feature of the proof is the new spectral approach used for the upper bound.
Recall that the normalized adjacency matrix of a graph without isolated vertices has largest eigenvalue \(1\); in the bipartite setting, the corresponding normalized biadjacency matrix has largest singular value \(1\). 
Consequently, the matrix associated with the non-leaf core, being a submatrix of the full normalized biadjacency matrix, has largest singular value at most \(1\). 
Our argument combines this standard spectral fact with the local average-degree condition to obtain a sharp global inequality. 
This suggests that normalized biadjacency matrices and their singular values may provide a useful tool for other extremal problems involving local degree constraints.

\section*{Declaration on the use of AI}
The authors used generative AI tools to assist in discussing proof strategies, checking proofs, and improving exposition. All mathematical arguments, results, and conclusions were reviewed and verified by the authors.

\bibliographystyle{abbrv}
\bibliography{bipartite}

\end{document}